\documentclass[12pt, a4paper]{amsart}
\usepackage[utf8]{inputenc}
\usepackage[greek, british]{babel}
\usepackage{microtype}
\usepackage{textcomp}
\usepackage{lmodern, fixcmex}
\usepackage[scr = dutchcal, scrscaled = 1.05]{mathalfa}
\usepackage{indentfirst}

\usepackage[
backend = biber, 
style=alphabetic,
citestyle=alphabetic,
sorting=nyvt
]{biblatex}
\AtBeginBibliography{\footnotesize}

\usepackage{graphicx, tikz, tikz-cd}
\usetikzlibrary{decorations.markings,arrows,arrows.meta, patterns, shapes}
\usepackage{amsmath, mathtools, amssymb, amsfonts, amsthm}
\usepackage[inline]{enumitem}
\usepackage{calc}
\usepackage{geometry}
\usepackage[
colorlinks=true, 
urlcolor=blue!50!black, 
linkcolor=black, 
citecolor=black, 
raiselinks=false,
pdfa
]{hyperref}
\usepackage[
english = british,
italian = quotes
]{csquotes}
\usepackage[unicode]{genstylemini}
\usepackage[e]{esvect}

\setlist[enumerate]{
	topsep = 1ex,
	itemsep=0ex,
	label=\arabic*.,
	ref  =\arabic*.
}
\setlist[itemize]{
	itemsep=0ex,
	topsep = 1ex
}

\newtheorem{theo}{Theorem}[section]
\newtheorem{lemma}[theo]{Lemma}
\newtheorem{corol}[theo]{Corollary}
\newtheorem{propos}[theo]{Proposition}
\newtheorem{fact}[theo]{Fact}

\theoremstyle{definition}
\newtheorem{defin}[theo]{Definition}

\newtheorem{question}[theo]{Question}

\newtheorem{claim}[theo]{Claim}

\theoremstyle{remark}
\newtheorem{obs}[theo]{Remark}

\newcommand{\indep}[2][]{%
	\mathrel{
		\mathop{
			\vcenter{
				\hbox{\oalign{\noalign{\kern-.3ex}\hfil\scalebox{1}[.8]{$\vert$%
							\makebox[0pt]{\raisebox{.9ex}{\kern.3ex\scalebox{.5}{$\mathsf{#2}$}}}}%
						\hfil\cr
						\noalign{\kern-.7ex}
						$\smile$\cr\noalign{\kern-.3ex}}}
			}
		}\displaylimits_{#1}
	}
} 
\newcommand{\nindep}[2][]{%
	\mathrel{
		\mathop{
			\vcenter{
				\hbox{\oalign{%
						\noalign{\kern-.3ex}\hfil\scalebox{1}[.8]{$\vert$%
							\makebox[0pt]{\raisebox{.9ex}{\kern.3ex\scalebox{.5}{$\mathsf{#2}$}}}\makebox[0pt]{\kern-.7ex$\smallsetminus$}}%
						\hfil\cr
						\noalign{\kern-.7ex}
						$\smile$\cr\noalign{\kern-.3ex}}}
			}
		}\displaylimits_{#1}
	}
} 

\makeatletter
\newcommand{\dotminus}{\mathbin{\text{\@dotminus}}}
\newcommand{\@dotminus}{%
	\ooalign{\hidewidth\raise.9ex\hbox{.}\hidewidth\cr$\m@th-$\cr}%
}
\makeatother
\let\supp\relax
\DeclareMathOperator{\supp}{supp}

\DeclareRobustCommand\models{\mathrel{\mathrel{|}\joinrel\mkern-.7mu\mathrel{=}}}

\newcommand{\lowersub}[1]{\smash{_{{\footnotesize\textstyle\mathstrut}#1}}}

\newcommand{\sametype}[1][]{\equiv_{#1}}
\DeclareMathOperator{\type}{tp}

\newcommand{\meet}{\mathbin{\wedge}}
\newcommand{\join}{\mathbin{\vee}}
\newcommand{\uint}{\,\mathbb{I}} 
\newcommand{\tsigma}{\textgreek{\textsigma}} 
\newcommand{\ind}[1][]{\indep[#1]{}}
\newcommand{\sigmind}[1][]{\indep[#1]{\ \scalebox{1.4}{$σ$}}}

\newcommand{\gdelta}{\textup{G\textsubscript{\textgreek{d}}}}

\newcommand{\dd}{\,\mathrm{d}}
\newcommand{\disjoint}{\mathbin{\perp}}
\DeclareMathOperator{\acl}{acl}
\DeclareMathOperator{\dcl}{dcl}
\DeclareMathOperator{\bandsymb}{€b}
\newcommand{\band}[1]{\bandsymb§(#1)}
\newcommand{\lBL}{£L_{\mathrm{BL}}}
\newcommand{\alol}{\ensuremath{\mathrm{AL_1L}}}
\newcommand{\alpl}[1][p]{\ensuremath{\mathrm{AL_{\mathit{#1}}L}}}
\newcommand{\rokh}[1]{\ensuremath{{R}_{#1}}}

\newcommand{\seq}[1]{\harpoonhat{#1}}
\newcommand{\class}[1]{§[#1]}
\newcommand{\symii}{II}
\newcommand{\symiif}{II\textsubscript{$1$}}
\newcommand{\symiin}{II\textsubscript{$\infty$}}
\newcommand{\symiii}{III}
\newcommand{\symx}{X}
\newcommand{\kindii}{kind \symii}
\newcommand{\kindiif}{kind \symiif}
\newcommand{\kindiin}{kind \symiin}
\newcommand{\kindiii}{kind \symiii}
\newcommand{\wtopol}{\mathscr{w}} 
\newcommand{\typetop}{\mathscr{l}} 

\newgeometry{left=2.5cm,right=2.5cm,top=3cm,bottom=3cm} 
\numberwithin{equation}{section}

\renewcommand{\seq}[1]{#1}

\makeatletter 
\def\ps@firstpage{\ps@plain
	\def\@oddfoot{\normalfont\scriptsize \hfil\rule{0pt}{20pt}\thepage\hfil
		\global\topskip\normaltopskip}%
	\let\@evenfoot\@oddfoot
	\def\@oddhead{\@serieslogo\hss}%
	\let\@evenhead\@oddhead 
}
\makeatother

\begin{document}
	
	\title{Banach $L^p$ lattices with an automorphism}
	\author{Antonio~M.~Scielzo}
	\address{Antonio~M.~Scielzo, Université Claude Bernard Lyon 1, Institut Camille Jordan, CNRS UMR 5208, 43 boulevard du 11 novembre 1918, 69622 Villeurbanne Cedex, France}
	\email{scielzo@math.univ-lyon1.fr}
	
	\begin{abstract}
		We study the theory  of Banach $L^p$ lattices with a distinguished automorphism, in the framework of continuous logic. Using a functional version of the Rokhlin lemma, we prove that it admits a model companion, which is stable and has quantifier elimination.
		We show that the types of this theory that are not trivial cannot be isolated. We then use this result to obtain a proof of the absence of comeagre conjugacy classes in $\aut^*(μ)$, the Polish group of non-singular transformations of a standard probability space. 
	\end{abstract}

	\maketitle
	\tableofcontents
	
	\section{Introduction}
	
	The aim of this paper is to study the model-theoretic properties of Banach $L^p$ lattices equipped with an automorphism. We will work in the framework of continuous logic as described in \cite{cont-logic} and \cite{claoc}. 
	Given a measure space $(X,£F,μ)$ we denote by $L^p(X,£F,μ)$ the vector space of $p$-integrable functions modulo equality $μ$-almost everywhere. Together with the norm $\norm{f} = §(\int \abs{f}^p\dd μ){}^{1/p}$, this is a Banach space. Moreover, the order $⩽$ given by pointwise comparison is a lattice order on $L^p(X,£F,μ)$, and is compatible with the structure of normed vector space. The aggregate $(L^p(X,£F,μ), \norm{\,⋅\,}, ⩽)$ is what we call a Banach $L^p$ lattice.

	For any given $0⩽ p < \infty$,	\alpl\ will be the theory of these lattices. 
	Model theoretic properties of \alpl\ were studied by Ben Yaacov, Berenstein, and Henson in \cite{lpIndep}, where they prove that \alpl\ is stable, and give a characterization of non-dividing using concepts from analysis. It was already proved in \cite{ultraprods} that \alpl\ has quantifier elimination, and it follows from Kakutani representation theorem \cite{meyer-banach} that \alpl\ is separably categorical, meaning that there is only one separable Banach $L^p$ lattice up to isomorphism, namely $L^p(\ccint{0,1},£B,μ)$, with $μ$ the Lebesgue measure on $\ccint{0,1}$.
	
	A natural question in model theory is the following. Let $T$ be a theory in a given language $£L$. If we expand $£L$ with a function symbol $σ$ and define $T_{σ}$ to be the theory $T$ together with an axiom stating that $σ$ is an automorphism, does $T_{σ}$ admit a model companion $T_A$? Unfortunately, there is no general criterion for the existence of such a theory $T_A$.
	Here we show that the theory \alpl\ expanded with an automorphism admits a model companion. This is a generalisation of the analogous result in \cite{autProbSpace}, where they prove that the theory of probability algebras with an automorphism has a model companion.
	
	Under certain conditions, the Banach lattice automorphisms of $L^p(X,£F,μ)$ correspond precisely to the non-singular transformations of $(X,£F,μ)$, that is, invertible measurable maps that preserve the family of negligible sets. These transformations generalise the concept of measure-preserving maps and are at the heart of non-singular ergodic theory. A survey on the main results concerning non-singular dynamical systems has been written by Danilenko in \cite{danilenko}.

	This close connection between Banach lattice automorphisms and non-singular transformations allows us to investigate some dynamical properties of $L^p$ lattices as well as their types. The model-theoretic results thus obtained are then used to prove the absence of comeagre conjugacy classes in the Polish group of non-singular transformations of the unit interval, 	
	in what appears to be an interesting connection between the two areas of mathematics.

	This paper is organized in the following way. 
	In \Cref{section: aut}, we introduce the basic notions concerning $L^p$ lattices and their automorphisms. We will recall Kakutani representation theorem, which allows us to identify abstract $L^p$ lattices with concrete structures $L^p(X,£F,μ)$. We then extend the representation to the automorphisms of $L^p$ lattices, in the separable case, showing that they correspond to non-singular transformations of a standard probability space, which are aperiodic precisely when they satisfy the Rokhlin lemma. We then provide a functional version of the Rokhlin lemma.

	In \Cref{section: TA}, we introduce the theory of atomless $L^p$ lattices with a distinguished automorphism and show that it admits a model companion $T_A$, which answers a question raised in \cite{autProbSpace}. The main tool here will be  the functional Rokhlin lemma.  We then use a result by Lascar \cite{lascar} to show that $T_A$ has quantifier elimination, and follow the same idea as in \cite{chatpill} to characterise the independence in $T_A$ and prove that $T_A$ is stable.
	
	In \Cref{section: types}, we recall the definition of the logic topology and of the metric for the space of types $S_n(T_A)$. We then prove that in $S_1(T_A)$ there are no non-trivial isolated types, i.e., points in the space where the two topologies coincide. We then introduce a notion of ergodicity for $L^p$ lattices that corresponds to the measure theoretic one in the separable case. In ergodic theory, non-singular ergodic transformations can be classified based on the existence of finite or \tsigma-finite equivalent invariant measures. We present here an analogous classification of ergodic lattices based on the types they realise.
	
	Finally, in \Cref{section: conjugacy}, we present an application of the absence of non-trivial isolated $1$-types of $T_A$. We recall the definition of the weak topology for the group $\aut^*(μ)$ of non-singular transformations of a standard probability space, which makes it a Polish group. We then consider a separable $L^p$ lattice $E$ and identify the automorphisms $σ$ that make $(E,σ)$ a model of $T_A$ with the aperiodic transformations in $\aut^*(μ)$, which form a comeagre subset of the group. We then prove that if an automorphism $σ$ has a comeagre conjugacy class, then $(E,σ)$ omits all non-isolated types, which is impossible.

	\section{Automorphisms of $L^p$ lattices} \label{section: aut}

	We start with some basic definitions following the presentation and notation from \cite{meyer-banach}.
	We say that a real Banach space $E$ together with a lattice order $⩽$ is a \emph{Banach lattice} if for all $u,v,w ∈ E$, we have
	\begin{itemize}[nosep, label={---}]
		\item (translation invariance)\; $u ⩽ v$ implies $u + w ⩽ v + w$,
		\item (positive homogeneity)\; for any scalar $0 ⩽ r$, if $u ⩽ v$ then $ru ⩽ rv$,
		\item (monotonicity)\; $\abs{u} ⩽ \abs v$ implies $\norm u ⩽ \norm v$,
	\end{itemize}
	where $\abs{x} = \sup §{x,-x} = x \join (-x)$. Two elements $u$ and $v$ in $E$ are said to be \emph{disjoint} if $\abs{u} \meet \abs v = 0$, and given $1⩽p<\infty$, a Banach lattice $(E,⩽)$ is called an \emph{$L^p$ lattice} if $\norm{u+v}^p = \norm{u}^p + \norm{v}^p$ whenever $u$ and $v$ are disjoint. A non-zero element of a Banach lattice that cannot be written as the sum of two other disjoint non-zero elements is called an \emph{atom}. If the lattice has no atoms, we say that it is \emph{atomless}. In the following, the lattices we deal with will always be atomless, unless otherwise specified.
	
	If $(X,£F,μ)$ is a measure space, then the space $L^p(X,£F,μ)$ of $p$-integrable functions modulo equality $μ$-almost everywhere, together with the order given by pointwise comparison, is an $L^p$ lattice. Kakutani representation theorem \cite[Theorem~2.7.1]{meyer-banach} states that every abstract $L^p$ lattice is in fact the concrete $L^p$ lattice of a measure space.
	In the separable atomless case, Kakutani representation theorem takes a more precise form.
	\begin{fact}[{\cite[Theorem~2.7.3]{meyer-banach}}]
		If $E$ is a separable atomless $L^p$ lattice, then it is isomorphic to $L^p(\uint,£B,λ)$, where $(\uint,£B,λ)$ is the Lebesgue measure space of the unit interval $\uint = \ccint{0,1}$.
	\end{fact}
	
	Let $E$ be a Banach lattice. A vector subspace $B$ of $E$ is called a \emph{band} if 
	\begin{itemize}[nosep, label={---}]
		\item for all $u∈E$ and $v ∈ B$, whenever $\abs{u} ⩽ \abs{v}$, we have $u∈B$,
		\item for every subset $A⊆B$ that has a supremum in $E$, we have $\sup(A) ∈ B$.
	\end{itemize}
	If $A$ is a subset of $E$, we will denote by $\band{A}$ the smallest band containing $A$. If $A = §{u}$ for some $u∈E$, we will write $\band{u}$ for the band generated by $§{u}$. A band that is generated by a single element is called a \emph{principal band}. 
	
	Given a set $A⊆E$, the \emph{disjoint complement} $A^{\perp}$ of $A$ is the set of those $u∈E$ that are disjoint from any element of $A$. It turns out that the band generated by $A$ is precisely the double complement $A^{\perp\perp} = (A^{\perp})^{\perp}$ of $A$.
	By \cite[Theorem~1.2.9]{meyer-banach}, if $E$ is an $L^p$ lattice and $A⊆E$, we can decompose $E$ as the direct sum $E = \band{A} \oplus A^{\perp}$. Given $u∈E$ and $B⊆E$ a band, we will denote by $u \restr B$ the projection of $u$ onto $B$ along $B^{\perp}$. 
	
	Given a Banach lattices $E$, a \emph{Banach lattice automorphism} of $E$ is an isometric linear automorphism preserving the lattice order, or equivalently the lattice modulus $\abs{\,⋅\,}$.

	We define the \emph{restriction} of a positive element to another by
	\begin{equation*}
		x \restr y ≔ \lim_{n} 2^n §(\frac{x}{2^n} \meet y).
	\end{equation*}
	For the general case, just set $x \restr y ≔ x^+ \restr \abs{y} - x^- \restr \abs{y}$. This coincides with the the projection of $x$ onto the band generated by $y$ (See \cite[Prop.~1.2.11]{meyer-banach}). In particular, in a concrete lattice $L^p(X,£F,μ)$ we have $f \restr g = f \restr\relax {\supp g}$ for all positive $f$ and $g$. In fact,
	\begin{align*}
		\norm{ f ⋅ χ_{\supp g} - 2^n §(\frac{f}{2^n} \meet g) } 
		&= \norm{ (f-2^n g)⋅ χ_{ §{x: 0<g(x)<f(x)/2^n} } }\\
		&⩽ \norm{ f ⋅ χ_{ §{x: 0<g(x)<f(x)/2^n} } } \to 0
	\end{align*}
	as $n$ goes to infinity, since the sequence of sets $§{x: 0<g(x)<f(x)/2^n}$ is decreasing with empty intersection. It is easy to see that if $σ$ is an automorphism of vector lattices, then $σ(x \restr y) = σx \restr σy$.

	\subsection{Representation of automorphisms of $L^p$ lattices}
	
	If $σ\colon (X,£F,μ) \to (X',£F')$ is a measurable map, we denote by $σ_*μ$ the \emph{pushforward measure} defined by $σ_*μ(A) = μ(σ^{-1}A)$ for all $A∈£F'$. As usual, if $ν$ is another measure on $(X,£F)$, we say that $μ$ and $ν$ are equivalent if, for all $A∈ £F$, $μ(A) = 0$ if and only if $ν(A)= 0$.

	\begin{defin}
		Let $(X,£F, μ)$ be a measure space.  A map $σ\colon X \to X$ is said to be a \emph{non-singular transformation} (or measure-class-preserving transformation) if it is an invertible measurable map such that $μ$ and $σ_*μ$ are equivalent.
	\end{defin}
	
	Non-singular transformations on $(X,£F, μ)$ form a group, which we will denote by $\aut^*(X,£F,μ)$, or simply $\aut^*(μ)$, when the measure space in question is clear from the context.

	Suppose now that $(X,£F, μ)$ is a \tsigma-finite measure space and let $σ∈ \aut^*(μ)$. Then the Radon--Nikodym derivative of $\dd σ_*μ$ with respect to $\dd μ$ exists, so for each $u∈ L^p(X,£F,μ)$, we can define a measurable function $\widetilde{σ}(u)$ by
	\begin{equation}\label{action def}
		\widetilde{σ}(u) = §( \frac{\dd σ_*μ}{\dd μ} )^{\frac 1 p} ⋅ §(u \comp σ^{-1}).
	\end{equation}
	This is in fact a $p$-integrable function, since
	\begin{align*}
		\norm{\widetilde{σ}(u)}^p &= \int §( \frac{\dd σ_*λ}{\dd λ} ) ⋅ §(\abs{u}^p \comp σ^{-1}) \dd λ \\
		&= \int §(\abs{u}^p \comp σ^{-1}) \dd σ_*λ = \int \abs{u}^p \dd λ = \norm {u}^p.
	\end{align*}
	This also shows that $\widetilde{σ}$ is in fact an isometry of $L^p(μ)$, and it is easy to see that $\widetilde{σ}\colon L^p(μ) \to L^p(μ)$ is a vector lattice automorphism. This means that $\widetilde{σ}$ belongs to $\aut(L^p(μ))$, the group of Banach lattice automorphisms, which we will simply call \emph{automorphisms} in the following.
	
	Straightforward calculations show that the application $σ \mapsto \widetilde{σ}$ defines an action of $\aut^*(μ)$ on $L^p(μ)$ by automorphisms. Moreover, in the particular case where the measure space is the unit interval $\uint$ with its Lebesgue probability measure $λ$, every automorphism of $L^p(λ)$ is induced by a non-singular transformation of $\uint$, and the correspondence is a group isomorphism.
	
	\begin{fact}[{\cite[Thorem~2.4]{metricImaginaries}}]\label{representation aut separable case}
		The map $σ\mapsto \widetilde{σ}$ defined above is an isomorphism from $\aut^*(λ)$ to $\aut(L^p(λ))$. Its inverse is given by
		\begin{equation*}
			\overline{τ}(r) = \inf §{q ∈ \uint ∩ \qq : r ∈ \supp τ(χ_{\ccint{0,q}})} ,
		\end{equation*}
		for all $τ ∈ \aut(L^p(λ))$ and all $r∈ \uint$.
	\end{fact}
	
	In general however, not every automorphism of $L^p(μ)$ is induced by a non-singular transformation of $(X,£F,μ)$.

	\subsection{Rokhlin lemma for Banach lattice automorphisms}\label{rokhlin section}
	
	Let $(X,£B,μ)$ a standard Borel space with a \tsigma-finite measure.
	
	\begin{defin}
		A non-singular transformation $σ$ of $(X,£B,μ)$ is said to be \emph{aperiodic} if, for every integer $n>0$, the set $ §{x∈X : σ^n(x) = x }$ of $n$-periodic points is negligible.
	\end{defin}

	A basic example of an aperiodic transformation is the \emph{translation} $τ_{r}$ of step $r$, defined on the real line $(\rr, £B, λ)$ with its Lebesgue measure by
	\begin{equation}\label{def translation}
		τ_{r} (x) = x+r,
	\end{equation} 
	for any $r ∈ \rr$. 
	Another example is the rotation $ρ_{α}$ of angle $2πα$ defined on the unit interval $(\uint, £B, λ)$ with its Lebesgue measure by
	\begin{equation}\label{def rotation}
		ρ_{α}(x) = \begin{cases}
			x + α & \text{if } x+α⩽1\\
			x + α -1 & \text{otherwise},
		\end{cases}
	\end{equation}
	where $0< α< 1$ is an irrational number.
	
	Let $σ$ be a non-singular transformation of $(X,£B,μ)$. We recall now a classical statement of ergodic theory that we will reproduce in the context of $L^p$ lattices.
	
	\begin{fact}[Rokhlin lemma for non-singular transformation {\cite[Lemma~7.9]{friedman}}]
		If $σ$ is aperiodic, then for any $n>1$ and $ε>0$, there exists $A∈£B$ such that $σ^i A$, for $0⩽ i < n$ are pairwise disjoint and
		\begin{equation*}
			μ§(X \setdiff \bigcup_{i<n} σ^i A)[Big] < ε.
		\end{equation*}
	\end{fact}
	
	The following is more basic result concerning aperiodic transformations, which can be obtained by repeated applications of \cite[Lemma~7.1]{friedman}.
	\begin{lemma}\label{basic lemma}
		For every $A∈£B$ of positive measure and every integer $N>0$, there exists $B⊆A$ of positive measure such that $σ^k(B) ∩ B = ∅$ for all $1⩽k⩽N$.
	\end{lemma}

	It turns out that the Rokhlin lemma characterises aperiodicity of non-singular transformations, as we can see in the following proposition.
	
	\begin{propos}
		Suppose that $σ$ satisfies the Rokhlin lemma, i.e., for all $n>0$ and $ε>0$, there is a measurable set $A$ such that the sets $A,σA,…,σ^{n-1}A$ are pairwise disjoint and together cover all of $X$ except at most a portion of measure $ε$. Then $σ$ is aperiodic.
	\end{propos}
	\begin{proof}
		Suppose it is not. Then there is some $m∈\nn$ such that $Y = §{x∈X : σ^m x = x}$ has positive measure. Choose $n = 2m$ and $ε = μY/2$ and find a set $A$ satisfying Rokhlin's condition. As $μY > 2ε$, there must be a $k<n$ such that $σ^k A$ intersects $Y$, so there is $x ∈ A$ such that $σ^k x ∈ Y$.
		Let $m'$ be either $m$ or $-m$ according to which one makes the inequality $0⩽ m'+k < n$ hold. Now, $σ^kA$ and $σ^{m'+k}A$ are disjoint, so $σ^k x ≠ σ^{m'+k} x$, but $σ^{m'} (σ^k x) = σ^{k} x$ by choice of $Y$, a contradiction.
	\end{proof}
	
	The following is a version of the Rokhlin lemma for $L^p$ lattices, which will be our main tool in proving the existence of the model companion of $T_{σ}$. 
	We write here a direct proof, due to Itaï Ben Yaacov. Here we are assuming that the lattice $L^p(μ) = L^p(\uint,£B,μ)$ is equipped with a an automorphism $σ$ induced by an aperiodic non-singular transformation of $\uint$, which is denoted by the same letter $σ$.
	
	\begin{lemma}[Functional Rokhlin lemma]\label{functional rokh}
		Let $f$ be a positive function in $L^p(μ)$. For every $n∈\zz_{>0}$ and every $ε > 0$,
		there exists a positive $g ∈ L^p(μ)$ of norm no greater than $\norm{f}$ and a positive  $h$ of norm at most $ε$, such that $σ^k(g)$ are pairwise disjoint for $0 ⩽ k < n$ and $f ⩽ \sum_{k<n} σ^k(g) + h$.
	\end{lemma}
	\begin{proof}
		We will use the following claim, which is a generalisation of \Cref{basic lemma}.
		\begin{claim}
			For every $A⊆ \uint$ of positive measure and positive integer $N$, there exists $B⊆A$ of positive measure such that $σ^k(B) ∩ B = ∅$ for all $1⩽k⩽N$ and in addition $A ⊆ \bigcup_{-N⩽ k ⩽ N} σ^k(B)$ up to negligible sets.
		\end{claim}
		\begin{proof}\renewcommand*{\qedsymbol}{$\triangle$}
			Construct an increasing sequence $(B_{α})_{α}$ of subsets of $A$ such that $σ^k(B_{α}) ∩ B_{α} = ∅$ for all $1⩽k⩽N$. Start with $B_0 = ∅$. For a limit ordinal $α < ω_1$, let $B_{α} = \bigcup_{β <α} B_{β}$. Given $B_{α}$, let
			\begin{equation*}
				A_{α} = A \setdiff \bigcup_{-N⩽k⩽N} σ^k(B_{α}).
			\end{equation*}
			If $A_{α}$ is negligible, then we may stop and choose $B = B_{α}$. Otherwise, apply \Cref{basic lemma} to $A_{α}$ to obtain a set $C_{α} ⊆ A_{α}$ disjoint from its first $N$ images under $σ$, and set $B_{α+1} = B_{α} ∪ C_{α}$. As the $C_{α}$ are non-negligible, the construction must stop as some countable ordinal.
		\end{proof}
		Now fix $N ⩾ n\norm{f}/ε$ and apply the claim to find a set $B$ such that $σ^k(B) ∩ B = ∅$ for $1 ⩽ k ⩽ N$ and $\uint ⊆^* \bigcup_{-N ⩽ k ⩽ N} σ^k B$, i.e., the union covers all of $\uint$ except at most a negligible set. By applying $σ^{-N-1}$ to both sides, we then have $B ⊆^* \bigcup_{-2N-1 ⩽ k ⩽ -1} σ^k B$, so that, for each $x∈ B$ there is a positive $k ⩽ 2N+1$ such that $σ^k(x) ∈ B$. Call $k(x)$ the least such $k$ and 
		define $B_{\ell} = §{x∈ B : k(x) = \ell}$. These sets are disjoint and measurable, and $B ⊆^* \bigcup_{N+1⩽ \ell ⩽ 2N+1} B_\ell$. Moreover, the sets $σ^k B_\ell$ with $N+1⩽ \ell ⩽ 2N+1$ and $0⩽ k < \ell$ form a partition of $\uint$ (up to a negligible set). 
		
		Let us first assume that $S_\ell = \bigcup_{0⩽ k < \ell} σ^k B_\ell$ supports $f$ for some $\ell$, and let $f_k$ be the restriction of $f$ to $σ^k B_\ell$, so $f = \sum_{k<\ell} f_k$. Now write $\ell = qn + p$, with $0⩽ p < n$, and notice that there must be some $q_0⩽ q$ such $\sum_{m<p} \norm{f_{q_0n+m}} ⩽ \norm{f} / (q+1)$. We are now going to “split” the construction, avoiding indices between $q_0n$ and $q_0n + p$. For each $m< n$, we define 
		\begin{align*}
			g_m = \sum_{0⩽k<q_0} f_{kn + m} \; + \sum_{q_0⩽k<q} f_{kn + p + m}.
		\end{align*}
		As the $f_k$'s are pairwise disjoint, the $g_m$'s are too. In addition,
		\begin{equation*}
			\norm{h} ⩽ \frac{\norm{f}}{q+1} ⩽ \frac{\norm{f}n}{N} ⩽ ε
		\end{equation*}
		by the choice of $N$, and $f = h + \sum_{m<n} g_m$. Finally, let
		\begin{equation*}
			g = \sum_{m<n} σ^{-m} g_m.
		\end{equation*}
		Then $\norm{g} ⩽ \norm{f}$, and $σ^m g$ shares support with $g_m$, so it is disjoint from $g$. Moreover, $σ^m g ⩾ g_m$, which means that $f = h + \sum_{m<n} g_m ⩽ h + \sum_{m<n} σ^m g$.
		
		In the general case, just handle each $\ell$ separately: define $f^{(\ell)} = f \restr S_{\ell}$, find the corresponding $g^{(\ell)}$ and $h^{(\ell)}$, and then take the sum over $\ell$.
	\end{proof}
	
	Roughly speaking, this lemma states that up to an arbitrarily small error, every positive $p$-integrable function $f$ is bounded by the sum of a given number of disjoint $σ$-images of another positive $p$-integrable function $g$, which is not greater than $f$ in norm.

	\section{Model theory of $L^p$ lattices with an automorphism}\label{section: TA}
	
	We assume that the reader is familiar with the basic notions of continuous logic, which can be found for instance in \cite{cont-logic}.  We will follow however the slightly different convention for formulas present in \cite{claoc}, where the family of all formulas is closed under uniform convergence or, more precisely, forced limit, so will not make a distinction between definable predicates and formulas. If we drop the forced limit construct, we obtain the so-called \emph{basic formulas}.
	
	The class of atomless $L^p$ lattices is elementary in the continuous language $£L_{\mathrm{BL}} = §{0,-, \frac{x+y}{2}, \abs{\,⋅\,}, \norm{\,⋅\,}}$ and a complete axiomatisation may be found in \cite[\S2]{nakano}. We shall denote this theory by \alpl. To be precise, a model of \alpl\ is just a closed ball of an $L^p$ lattice, but this is enough to recover the entire lattice, so we will not make any distinction in the following and will still call these balls $L^p$ lattices.
	
	Model theoretic properties of \alpl\ were studied by Ben Yaacov, Berenstein, and Henson in \cite{lpIndep}, although with a different, but equivalent formalism. In \cite[Proposition~4.11 and Theorem~4.12]{lpIndep} they give a characterisation of independence for $L^p$ lattices and show that \alpl\ is stable. The fact that \alpl\ has quantifier elimination was already proved in \cite{ultraprods}. In addition, it follows from Kakutani representation theorem that \alpl\ is separably categorical, which in turn implies that its separable model $E$ is approximately homogeneous, meaning that, if $a$ and $b$ are tuples in $M$ with the same type, then for every $ε>0$ there exists an automorphism of $E$ such that $d(fa,b) < ε$.
	
	Write $£L_{σ}$ for the language of Banach lattices $£L_{\mathrm{BL}}$ expanded with a unitary function symbol $σ$ and define $T_{σ}$ to be the theory \alpl\ together with the axioms stating that $σ$ is an automorphism of Banach lattices, that is,
	\begin{itemize}[nosep, label={---}]
		\item (morphicity) $\sup_{\bar x} \norm{σ(F\bar x) - F(σ\bar x)} = 0$, for each function symbol $F$ in $£L_{\mathrm{BL}}$,
		\item (isometry) $\sup_{x} \abs{\norm{σx} - \norm{x}} = 0$,
		\item (sujectivity) $\sup_x \inf_y \norm{x-σy} = 0$.
	\end{itemize}

	For each $n>0$, let the $n$-th \emph{Rokhlin axiom} be the continuous sentence
	\begin{equation}\label{rokhlin axiom}
		\adjustlimits \sup_{x} \inf_{y} 
		\max§(
		\norm{ σ^i\abs{y} \meet \abs{y} }[big] : 0<i<n,
		\norm{ §( \sum_{k<n} σ^k \abs{y} - \abs{x} )[Big] ^{\!-} }[bigg],
		\norm{y} \dotminus \norm{x}
		)[bigg].
		\tag{\rokh{n}}
	\end{equation}
	Then $R_n = 0$ means that for all positive $x$ and all $ε>0$, there is a positive $y$ of norm less than $\norm{x}+ε$, whose first $n$ images under $σ$ are disjoint up to an error $ε$, and their sum is greater than $x$ except for a portion of norm less than $ε$.
	In the following, we will also use the notation $\rokh{n}(x,y)$ to mean $\rokh{n}$ with both quantifiers removed. 
	
	We call $T_A$ the theory $\mathrm{AL_1L}_{σ}$ together with all Rokhlin's axioms. \Cref{functional rokh} provides some examples of models of $T_A$; in fact, it characterises the separable models of $T_A$.
	
	\begin{lemma}
		If $σ$ is an aperiodic non-singular transformation of a standard atomless probability space $(X,£B,μ)$, then $(L^p(μ),\widetilde{σ})$ is a model of $T_A$.		
	\end{lemma}

	We will split the proof that $T_A$ is the model companion of $T_{σ}$ into several lemmas. Clearly, every model of $T_A$ is a model of $T_{σ}$, so we just need to show that every model of $T_{σ}$ embeds in a model of $T_A$ and that $T_A$ is model complete. As  \alpl\ is definable in \alol\ by \cite[Lemma~3.3]{canonicalBases}, we will carry out these proofs assuming $p=1$, but the results will hold for any $p⩾1$.
	
	\begin{lemma}
		Every model of $T_{σ}$ embeds in a model of $T_A$.
	\end{lemma}
	\begin{proof}
		It is enough to prove this for separable models.
		Let $(M,σ)$ be a model of $T_{σ}$. By	\Cref{representation aut separable case}, we may assume that $M$ is the $L^1$ lattice over the unit interval $L^1(λ)$ and $σ$ is induced by a non-singular transformation $f\colon \uint \to \uint$ by $σ(u) = \widetilde{f}(u) = (\!\dd σ_*λ/\!\dd λ)(u \comp f^{-1})$.
		
		Now let $g$ be an aperiodic measure-preserving transformation of $\uint$, for instance an irrational rotation, as defined in \eqref{def rotation}.  Then define a transformation $h$ of the unit square $(\uint^2,£B^2,λ^2)$ by
		\begin{equation*}
			h (x,y) ≔ §( f(x), g(y) )
		\end{equation*}
		and notice that $h_*λ^2 = f_*λ \otimes g_*λ$, so that $h$ is also non-singular. Additionally, $h$ is aperiodic, since for all $n$
		\begin{align*}
			λ^2 §{ (x,y): h^n (x,y) = (x,y) } 
			&= λ^2 §{ (x,y): f^n (x) = x \text{ and } g^n (y) = y } \\
			&= λ §{ x: f^n (x) = x} ⋅ \xunderbrace{λ §{y: g^n (y) = y }}{=0} = 0,
		\end{align*}
		so the induced automorphism $τ = \widetilde{h}$ on $N= L^1(λ^2)$ satisfies Rokhlin's axioms. Consequently, we just need to check that the application $Φ\colon M \to N$ defined by $Φ(u)(x,y) = u(x)$ is an $£L_{σ}$-embedding. It clearly is an isometry of Banach lattices, so it remains to show that $τ \comp Φ =   Φ \comp σ$, but
		\begin{align*}
			τ(Φ(u))(x,y) &= \frac{\dd h_*λ^2}{\dd λ^{2}}(x,y) ⋅ Φ(u)§( h^{-1}(x,y) )
			\\
			&= \frac{\dd {f}_*λ}{\dd λ}(x)
			⋅ \xunderbrace{\frac{\dd {g}_*λ}{\dd λ}(y)}{=1} 
			⋅ Φ(u)( f^{-1}x , g^{-1}y )
			\\
			&= \frac{\dd {f}_*λ}{\dd λ}(x)
			⋅ u( f^{-1}x ) = σ(u)(x) = Φ(σ(u))(x,y),
		\end{align*}
		which concludes the proof.
	\end{proof}

	We will use the following characterisation of model completeness in continuous logic.
	
	\begin{fact}[{\cite[Exercise~6.21]{claoc}}]\label{model-completeness criterion}
		A theory $T$ is model-complete if and only if, for all $\aleph_1$-saturated 
		models $M ⊆ N$ of $T$, every quantifier-free basic formula $φ(x,y)$, every $c∈M^{\card{y}}$, every $b∈N^{\card{x}}$, and every $ε>0$, there exists $a∈M^{\card{x}}$ such that
		\begin{equation*}
			\abs{ φ(a,c) - φ(b,c) } < ε.
		\end{equation*}
	\end{fact}

	By assuming the model to be saturated we have an exact version of the Rokhlin lemma, with no error. We will denote the band generated by $f$ by $\band{f}$.
	
	\begin{lemma}\label{exact rokh}
		Suppose $M$ is an $\aleph_1$-saturated model of $T_A$ and $F$ is a finite set of positive elements of $M$. Then there exists a principal band $B$ which contains $F$ and is invariant under $σ$, meaning that $σ(B) = B$. Moreover, for any integer $n>0$, there is a positive $g$ in $M$ such that
		\begin{math}
			§{ \band{σ^i g} : i < n }
		\end{math}
		is a partition of $B$. In particular, $g$ and $σ^n g$ generate the same band.
	\end{lemma}
	\begin{proof}
		Consider the average $f_0 ≔ \frac 1{\card{F}}\sum_{f∈F}\abs{f} ∈ M$ and define
		\begin{equation*}
			f_1 ≔ \frac{1}{3} \sum_{k∈\zz} \frac{σ^k f_0}{2^{\abs{k}}},
		\end{equation*}
		which is still an element of $M$ and generates a band $B ≔ \band{f_1}$ that is invariant under $σ$ and contains $F$.
		
		By \rokh{n} applied to $f_1$ and saturation, 
		there exists a positive $g∈M$ such that $σ^i g \disjoint g$, for all $i<n$, and $f_1 ⩽ \sum_{i<n} σ^i g$. If we replace $g$ by $g\restr f_1$, we get an element still satisfying these properties but also lying in the band generated by $f_1$, so the bands $\band{σ^i g}$ for $i<n$ form a partition of $B$.
		
		As $σ^n g$ is disjoint from any $σ^i g$ with $0<i<n$ (because $σ^{n-1} g \disjoint σ^{i-1} g$ implies $σ^{n} g \disjoint σ^{i} g$) and by the invariance of $B$ under $σ$, we have
		\begin{equation*}
			\bigcup_{i<n} \band{σ^i g} = B = σ§( B ) = \bigcup_{0<i⩽n} \band{σ^i g} ,
		\end{equation*}
		where the unions are disjoint. This means that $\band{σ^n g} = \band{g}$.
	\end{proof}
	
	\begin{lemma}
		$T_A$ is model complete.
	\end{lemma}
	\begin{proof}
		We make use of \Cref{model-completeness criterion}. Consider then two $\aleph_1$-saturated models $M⊆ N$ of $T_A$, a quantifier-free basic formula $φ(x,y)$, some elements $f∈M^{\card{y}}$ and $h∈N^{\card{x}}$, and a positive number $ε$. For the sake of simplicity, we may assume that $\card{x} = 1$. We shall find $\hbar ∈ M^{\card{x}}$ such that $\abs{ φ(\hbar,f) - φ(h,f) } < ε$.
		Being quantifier-free, the formula $φ$ is of the form
		\begin{equation*}
			φ(x,f) = ψ(σ^i x : i < \ell; f_j : j < m ) ,
		\end{equation*}
		where $σ$ does not appear in $ψ$, possibly adding some new parameters, which we will write collectively again as $f$ in what follows.
		We will also abbreviate $(σ^{i} x : i< \ell)$ as $(σ^{<\ell} x)$.

		Now, the formula $ψ$ is a continuous combination of norms of terms $\norm{t(σ^{<\ell}x;f)}$, so there exists $δ_0>0$ such that 
		\begin{equation*}
			\max_{t \textrm{ term of } ψ} \abs{
				\norm{ t(σ^{<\ell} h; f)} - \norm{ t(σ^{<\ell} \hbar; f)}
			}[Big]
			< δ_0 
			\implies 
			\abs{ φ(h,f) - φ(\hbar,f) } < ε,
		\end{equation*}
		for all $\hbar ∈ M$. It will then suffice to find $\hbar ∈ M$ such that 
		$\abs{ \norm{ t(σ^{<\ell} h; f)} - \norm{ t(σ^{<\ell} \hbar; f)} }< δ_0$  for any term $t$ of $φ$. We will now split the space in many bands where the restrictions of most of the terms above in $h$ have the same norm as the corresponding term in $\hbar$. 
		
		It will then suffice to show that the remaining terms give a small contribution.
		As $\norm{t(\;\cdot\;;f)}{}^N$ is uniformly continuous, there is some $δ_1>0$ such that, for all $a_1,a_2∈N^{\ell}$,
		\begin{equation*}
			d_1(a_1,a_2) <δ_1 \implies \abs{\norm{t(a_1;f)} - \norm{t(a_2;f)}}[Big] < \frac{δ_0}{2\ell},
		\end{equation*}
		where $d_1$ is the distance given by pointwise sum of the components.
		
		We apply \Cref{exact rokh} to $M$ and $F≔§{f_j: j<m}$ with 
		\begin{equation*}
			n > \frac{4\ell \norm{h}}{δ_1},
		\end{equation*}
		in order to find a positive $g_0∈ M$ such that $§{\band{σ^i g_0} : i<n}$ forms a partition of the $σ$-invariant band $B$ generated by $F$. We consider the components $h'$ and $h''$ of $h$ respectively in the band of $N$ generated by $B$ and in its disjoint complement, which are both $σ$-invariant.
		
		We then apply \Cref{exact rokh} to $N$ and $h''$ and find $g_1∈N$ such that $§{\band{σ^i g_1} : i<n}$ forms a partition of a band disjoint from $B$ and containing $h''$. 
		
		For each $i<n$, consider the $i$-th component $h_i ≔ h'\restr g_0 + h'' \restr g_1$ of $h$. There is $n_0 <n$ such that the $\ell$ components from $n_0$ together with the previous $\ell$ make up less than $2\ell/ n$ of the total norm of $h$, that is,
		\begin{equation*}
			\sum_{k<\ell} \norm{h_{(n_0 + k) \bmod n}} + \sum_{k<\ell} \norm{h_{(n_0 + n -1 - k) \bmod n}} < \frac{2 \ell \norm{h} }{n}
		\end{equation*}
		We then replace $g_0$ and $g_1$ by their $n_0$-th images under $σ$, so that the first and last $\ell$ components of $h$ contribute to the norm of $h$ by less than $2\ell \norm{h} / n$.
		
		Denote $h \restr σ^i g_0$ by $h'_i$ and $h \restr σ^i g_1$ by $h''_i$ and consider the type of $(σ^{-i}h'_i: i<n)$ over  $§{g_0, σ^{-i}f : i<n}$ in the reduct $N \restr \lBL$ (i.e., $N$ considered just as a Banach lattice, without any reference to the automorphism). As $\mathrm{AL_1L}$ is model complete, this is a type in $M \restr \lBL$, and by saturation, it is realised by some tuple $(\hbar'_i: i<n)$ in $M$.
		In particular, each $\hbar'_i$ lies in the band generated by $g_0$, so that it is disjoint from its first $n-1$ images under $σ$.
		We thus have the following decomposition
		\begin{equation*}
			t§(σ^{<\ell}h; f ) = 
			\sum_{i<n} t§(σ^{<\ell}h; f ) \restr σ^i g_0 + 
			\sum_{i<n} t§(σ^{<\ell}h; f ) \restr σ^i g_1,
		\end{equation*}
		where each addend is disjoint from the others.
		Now, for all $i<n$,
		\begin{equation*}
			t§(σ^{<\ell} h ; f ) \restr σ^i g_0 = t§(σ^k h'_{(i-k)\bmod n} : k< \ell; f \restr σ^i g_0)
		\end{equation*}
		and
		\begin{equation*}
			t§(σ^{<\ell} h ; f ) \restr σ^i g_1 = t§(σ^k h''_{(i-k)\bmod n} : k< \ell; 0).
		\end{equation*}
		
		As $M$ is $\aleph_1$-saturated, there exists a positive  $a∈M$ disjoint from $B$ and generating a band invariant under $σ$. Apply \Cref{exact rokh} to $M$ and $a$ and find $g_2∈M$ such that $§{\band{σ^i g_2} : i<n}$ forms a partition of $\band{a}$. The band $\band{g_2}$ is still an $L_1$-lattice, so there is a realisation $(\hbar''_i: i<n)$ of the type of $(σ^{-i}h''_i: i<n)$ in $N \restr \lBL$. 
		Each $σ^i \hbar''_i$ lies in the band $\band{σ^i g_2}$, so that they are pairwise disjoint. 
		Thus,
		\begin{equation*}
			\hbar ≔ \sum_{i<n} σ^i(\hbar'_i + \hbar''_i)
		\end{equation*}
		may be decomposed in the same way as $h$ by replacing $g_1$ with $g_2$, $h'_j$ with $σ^j\hbar'_j$, and $h''_j$ with $σ^j\hbar''_j$. 
		
		We now compare the norms of the components of the terms in $h$ and $\hbar$.
		For all $\ell ⩽ i <n$, 
		\begin{align*}
			\norm{ t( σ^i\hbar'_{i-k} : k<\ell ; f \restr σ^i g_0) } 
			&= \norm{ t( \hbar'_{i-k} : k<\ell ; σ^{-i}f \restr g_0 )  } \\
			&= \norm{ t( σ^{k-i} h'_{i-k} : k<\ell ; σ^{-i} f \restr g_0 ) } \\
			&= \norm{ t( σ^{k} h'_{i-k} : k<\ell ; f \restr σ^i g_0 ) },
		\end{align*}
		where the first and last equalities follow from the fact that $σ$ is isometric and commutes with all other symbols of $\lBL$, and the second equality is the application of the realisation of types.
		For the same reasons we have	
		\begin{align*}
			\norm{ t( σ^i\hbar''_{i-k} : k<\ell ;0) } 
			&= \norm{ t( σ^{k} h''_{i-k} : k<\ell ; 0 ) },
		\end{align*}
		which means that the norms of the components of index $i⩾ \ell$ cancel out and we are left only with the first $\ell$ components. More precisely,
		\begin{multline}\label{ta-complete-unruly-bits}
			\abs{ \norm{t§(σ^{<\ell} h; f )} - \norm{t§(σ^{<\ell} \hbar; f )} }[Big]
			⩽ \\
			\sum_{i<\ell} \abs{
				\norm{ t( σ^{k} h_{(i-k)\bmod n} : k<\ell ;f \restr σ^i g_0 ) }
				- 
				\norm{ t( σ^{k} \hbar_{(i-k)\bmod n} : k<\ell ;f \restr σ^i g_0 ) }
			}[Big],
		\end{multline}
		where $\hbar_i$ is simply $σ^i(\hbar'_i + \hbar''_i)$.

		By the choice of $n_0$, for each $i<\ell$,
		\begin{equation*}
			\sum_{k<\ell}\norm{ h_{(i-k)\bmod n} } = \sum_{k<\ell}\norm{ \hbar_{(i-k)\bmod n} } < \frac{2\ell \norm{h} }{ n }
		\end{equation*}
		so that
		\begin{equation*}
			\sum_{k < \ell} \norm{ h_{(i-k)\bmod n} - \hbar_{(i-k)\bmod n} } < \frac{4 \ell \norm{h} }{ n } < δ_1
		\end{equation*}
		and thus, by uniform continuity of $\norm{t(\;\cdot\;;f)}{}^N$, the sum in \eqref{ta-complete-unruly-bits} is strictly less than $δ_0$, which is what we wanted.
	\end{proof}
	
	We have thus proved the main result.
	\begin{theo}
		$T_A$ is a model companion of $T_{σ}$.
	\end{theo}

	\subsection{Quantifier elimination} Given a cardinal $κ>2^{\aleph_0}$, recall that a normed space structure is said to be \emph{$κ$-universal} if it is $κ$-strongly homogeneous and $κ$-saturated. We will work in a $κ$-universal model $£U$ of \alpl, for some large $κ$.
	
	We will show that $T_A$ has quantifier elimination using the following result by Lascar, which can be shown to hold in continuous logic as well.
	
	\begin{fact}[{\cite[Thereom~3.3]{lascar}}]
		Let $T$ be a stable theory, $£C$ a large universal model of $T$, and $M_0$, $M_1$ and $M_2$ elementary substructures of $£C$. If $M_0 \prec M_1,M_2$ and $M_1$ and $M_2$ are independent over $M_0$, then $M_1 ∩ M_2 = M_0$ and for all automorphisms $α$ of $M_1$ and $β$ of $M_2$ having the same restriction to $M_0$, the application $α ∪ β \colon M_1 ∪ M_2 \to M_1 ∪ M_2$ is elementary.
	\end{fact}

	\begin{theo}\label{qe of TA}
		$T_A$ has quantifier elimination.
	\end{theo}
	\begin{proof}
		As in classical first order logic, a model complete theory $T$ has quantifier elimination if and only if its universal part $T_{∀}$ has the amalgamation property. 
		By \cite[Lemma~6.24]{claoc}, if $T^*$ is the model companion of $T$, then their universal parts coincide, so in our case we just need to show that $(T_{σ})_{∀}$ has the amalgamation property. 
		
		As the models of $T_{∀}$ are precisely the substructures of models of $T$, 
		it suffices to check that, given models $(E_1,σ_1)$ and $(E_2,σ_2)$ of $T_{σ}$, and embeddings of $£L_{σ}$ structures $f_i\colon A \hookrightarrow E_i$ from a model $(A,σ)$ of $(T_{σ})_{∀}$, there is an $L^p$ lattice with automorphism $(E,τ)$ and $£L_{σ}$-embeddings $g_i\colon E_i \to E$ making the following diagram commute. 
		\begin{equation*}
			\begin{tikzcd}[sep = small]
				& E_1,σ_1 \arrow[rd,hook,"g_1"] \\
				A,σ \arrow[ru,hook,"f_1"] \arrow[rd,hook,"f_2"'] & & E,τ\\
				& E_2,σ_2 \arrow[ru,hook,"g_2"']
			\end{tikzcd}
		\end{equation*}
		
		Let $\mathbb{M}$ be a sufficiently universal model of $\alpl$. Since \alpl\ is stable, we may assume that $A ⊆ E_1,E_2 ⊆ \mathbb{M}$ with $σ_1\restr A = σ_2 \restr A = σ$, and that $E_1$ is (forking-)independent of $E_2$ over $A$. In particular, $E_1 ∩ E_2 = \acl A $, but it follows from \cite[Fact~3.11 and Lemma~3.12]{lpIndep} that this is precisely the Banach lattice generated by $A$. Therefore $σ$ extends uniquely to a Banach lattice automorphism of $\acl A$, which allows us to assume that $A$ is algebraically closed.
		
		At this point, we can proceed as in \cite[Thereom~3.3]{lascar} to show that the map $σ_1 ∪ σ_2$ on $E_1 ∪ E_2$ is elementary, so we can conclude by extending $σ_1 ∪ σ_2$ to an automorphism $τ$ of the $L^p$ lattice $E$ generated  by $E_1 ∪ E_2$.
	\end{proof}
	
	As the  only constant in the language $L_{\mathring{BL}}$ is $0$, which is fixed by all functions in $T_{A}$, we have the following corollary.
	
	\begin{corol}
		$T_A$ is complete.
	\end{corol}

	\subsection{Independence and stability} In this section we show that $T_A$ is stable using an argument similar to the one presented in \cite{chatpill}, but instead of showing the independence theorem to prove that $T_A$ is simple, we will prove that $T_A$ admits a stationary relation of independence, which implies that $T_A$ is actually stable.

	Let $(£U, σ)$ a large universal model of $T_A$. We shall write $σ^{\zz}a$ to mean $(σ^i a)_{i∈\zz}$, and similarly, $σ^{\zz} A$ is shorthand for $§{σ^i u : u ∈ A, i ∈ \zz}$
	
	\begin{lemma}\label{sigma-sametype char}
		Let $\seq{a}$ and $\seq{b}$ be two tuples of the same lengths in $£U$, and $C$ a small subset of $£U$. Then $\seq{a}$ and $\seq{b}$ have the same type over $C$ in the sense of $(£U,σ)$ if and only if $σ^{\zz}\seq{a}$ and $σ^{\zz}\seq{b}$ have the same type over $σ^{\zz}C$ in the sense of $£U$.
	\end{lemma}
	\begin{proof}
		Both \alpl\ and $T_A$ have quantifier elimination, so we just need to check the equality of quantifier-free types. Suppose that $σ^{\zz}\seq{a}$ and $σ^{\zz}\seq{b}$ have the same type over $σ^{\zz}C$ in the sense of $£U$, and let $φ(\seq{x},\seq{c})$ be a quantifier-free formula in $£L_{σ}(C)$ vanishing at $\seq{a}$. Then $φ$ is of the form $ψ(f_1 \seq{x}, …, f_k \seq{x}, g_1 \seq{c},…,g_m \seq{c})$, where the $f_i$'s and the $g_i$'s are powers of $σ$ an no other instance of $σ$ appears in $ψ$. This means that $ψ(\seq{y_1}, …, \seq{y_k}, \seq{g}\seq{c})$ belongs to $£L_{\mathrm{BL}}(σ^{\zz}C)$ and vanishes at $\seq{f} \seq{a}$, so it also vanishes at $\seq{f} \seq{b}$, which means that $φ(\seq{b}, \seq{c}) = 0$.
		For the converse, just repeat the same reasoning in reverse.
	\end{proof}
	
	In the following we will denote the type of $a$ over $C$ in the sense of $£U$ simply by $\type(a/C)$, and the respective type in the sense of $(£U,σ)$ by $\type^{σ}(a/C)$. We will also write $a \sametype[C]^{σ} b$ to mean $\type^{σ}(a/C) = \type^{σ}(b/C)$, and similarly without the $σ$.
	We will also denote by $\dcl A$ the definable closure of $A$ in the sense of \alpl, and set $\dcl_{σ}(A) = \dcl(σ^{\zz} A)$.
	
	\begin{lemma}
		$\dcl_{σ}(A)$ is the algebraic closure of $A$ in the sense of $T_A$.
	\end{lemma}
	\begin{proof}
		By \cite[Fact~3.11 and Lemma~3.12]{lpIndep}, $\dcl_{σ}(A) = \acl(σ^{\zz}A)$ is the Banach lattice generated by $σ^{\zz}A$. It is then clear that every element in it is algebraic over $A$ in the sense of $T_A$. For the converse, suppose $a$ is algebraic over $S = \dcl_{σ}(A)$ in the sense of $T_A$ and rewrite the proof of \cite[Lemma~3.6]{chatpill} using the characterisation \cite[Lemma~4.9]{cont-logic} of algebraic types in continuous logic.
	\end{proof}
	
	\begin{defin}
		Let $A$, $B$, and $C$ be small subsets of $£U$. We say that $A$ is \emph{$σ$-independent} of $B$ given $C$ if $\dcl_{σ}(AC)$ is forking independent of $\dcl_{σ}(BC)$ over $\dcl_{σ}(C)$. We will denote $σ$-independence by $\sigmind$.
	\end{defin}
	
	Notice that by \cite[Theorem~4.12]{lpIndep}, we have
	\begin{equation}\label{sigmind char}
		A \sigmind[C] B \iff σ^{\zz} A \ind[σ^{\zz} C] σ^{\zz} B.
	\end{equation}

	\begin{lemma}
		The relation $\sigmind$ satisfies the following properties, for arbitrary small subsets $A, B, C, D$ of $£U$.
		\begin{enumerate}
			\item Invariance under automorphisms of $(£U,σ)$.
			\item Symmetry: $A \sigmind[C] B$ if and only if $B \sigmind[C] A$.
			\item Transitivity: $A \sigmind[C] BD$ if and only if $A \sigmind[C] B$ and $A \sigmind[BC] D$.
			\item Finite character: $A \sigmind[C] B$ if and only if $\seq{a} \sigmind[C] B$ for all finite tuples $\seq{a} ⊆A$.
			\item Extension: there is $A' \sametype[C]^{σ} A$ such that $A' \sigmind[C] B$.
			\item Local character: for any finite tuple $\seq{a}$, there is a countable $B_0⊆B$ such that $\seq{a} \sigmind[B_0] B$.
			\item Stationarity: if $A \sametype[C]^{σ} D$, $A \sigmind[C] B$, and $D \sigmind[C] B$, then $A \sametype[BC]^{σ} D$.
		\end{enumerate}
	\end{lemma}
	\begin{proof}
		Invariance, symmetry and transitivity follow immediately from the equivalence \eqref{sigmind char}. For the finite character, notice that if $A \sigmind[C] B$, then, by monotonicity of $\ind$, we have $σ^{\zz}\seq{a} \ind[σ^{\zz} C] σ^{\zz} B$, that is, $\seq{a} \sigmind[C] B$, for any finite tuple $\seq{a} ⊆A$. Conversely, if $\seq{a} \sigmind[C] B$, for any finite tuple $\seq{a} ⊆A$, then by monotonicity and finite character of $\ind$, we have $σ^{\zz}\seq{a} \ind[σ^{\zz} C] σ^{\zz} B$ and thus $A \sigmind[C] B$.
		
		For extension of $\sigmind$, use the corresponding property of $\ind$ to find some $Φ ∈ \aut(£U/σ^{\zz}C)$ such that $Φ(σ^{\zz} A) \sigmind[\dcl_{σ}C] σ^{\zz} B$ and define $A_0 = Φ(A)$ and $σ_{0} = ΦσΦ^{-1}$, so that $σ_0^{\zz} A_0 \sigmind[\dcl_{σ}C] σ^{\zz} B$. Using the amalgamation property shown in the proof of \Cref{qe of TA}, $(σ_0^{\zz}A_0,σ_0)$ and $(σ^{\zz}B,σ)$ are contained in a model $(E,τ)$ of $T_A$, and thus, by saturation of $(£U,σ)$, we can find $A' \sametype[C]^{σ} A$ that is $σ$-independent of $B$ over $C$.
		
		For local character, suppose $\seq{a}$ is a finite tuple and for each $n∈\zz_{>0}$, use the corresponding property of $\ind$ to find a countable subset $B_n⊆ σ^{\zz} B$ such that $σ^{\ccint{-n,n}} \seq{a} \ind[B_n] σ^{\zz} B$. By finite character and transitivity,  we have $σ^{\zz} \seq{a} \ind[σ^{\zz} \bigcup_n \! B_n] σ^{\zz} B$. Therefore, the set $B_0 = \bigcup_n\! B_n ∩ B$ is countable and $\seq{a} \sigmind[B_0] B$.
		
		Finally, stationarity follows immediately from \Cref{sigma-sametype char} and \eqref{sigmind char}. In fact, the assumptions are equivalent to $σ^{\zz} A \sametype[σ^{\zz}C] σ^{\zz} D$,  $σ^{\zz}A \ind[σ^{\zz} C] σ^{\zz} B$, and $σ^{\zz}D \ind[σ^{\zz} C] σ^{\zz} B$, so by stationarity of $\ind$, we have $σ^{\zz} A \sametype[σ^{\zz}(CB)] σ^{\zz} D$, that is, $A \sametype[CB]^{σ} D$.
	\end{proof}
	
	\begin{theo}
		$T_A$ is stable and $σ$-independence coincide with forking independence.
	\end{theo}
	\begin{proof}
		It follows from the previous lemma and \cite[Theorems~1.51,2.8]{simplicity}~.
	\end{proof}

	\section{Types and their dynamical properties}\label{section: types}

	Let $S_n(T)$ denote the set of (complete) $n$-types in $T$ over the empty set. Recall that a formula $φ$ induces a function $\bar{φ}\colon S_n(T)\to \rr$ defined by assigning to each type $€p ∈ S_n(T)$ the unique $r∈\rr$ such that $φ(\seq{x}) -r $ belongs to $€p$. The initial topology with respect to the collection of the $\bar{φ}$'s, i.e., the coarsest topology that makes these functions continuous, is called the \emph{logic topology}, and renders $S_n(T)$ a compact Hausdorff space, see \cite[Theorem~7.5]{claoc}. Moreover, the sets $\doublesqbr{φ > 0} = §{€p∈S_n(T) : \bar{φ}(€p) >0}$, with $φ$ varying among the the basic formulas, form a basis for this topology.
	
	When $T$ is a complete theory, any two types $€p,€q∈ S_n(T)$ are realised in a universal model $£U$ of $T$, so we can define
	\begin{equation*}
		d(€p,€q) = \inf§{ d_\infty(\seq{a}, \seq{b}): \seq{a}, \seq{b} ∈ £U^n, \seq{a} \models €p, \seq{b} \models €q},
	\end{equation*}
	where $d_\infty(\seq{a}, \seq{b}) = \max_{i<n} d(a_i,b_i)$ and $d$ is the distance in $£U$.
	In \cite[Section~4.3]{cont-logic} it is shown that the newly defined $d$ is a complete metric on $S_n(T)$ that refines the logic topology. Furthermore, the two topologies coincide precisely when $T$ is separably categorical. This means that in the case of our theory $T_A$ the metric is strictly stronger than the logic topology, at least globally. We will later see what happens at a local level.
	
	In the following we will use terms typically associated to metric spaces to refer to the type metric, while the other topological terms, such as open and closed sets, neighbourhoods, interiors, 	
	unless otherwise specified.

	\begin{obs}
		Let $€p$ and $€q$ be type of a model of $T_A$ such that $\norm{x^-}*{}^{€p} = \norm{x^-}*{}^{€q} = 0$.
		For every $n<ω$ we have
		\begin{equation}\label{lower bound on type distance}
			d(€p,€q) ⩾ \abs{ 
				\norm{x}^{€p} - \norm{x}^{€q} + \norm{x \meet σ^n x}^{€p}- \norm{x \meet σ^n x}^{€q} 
			}[Big] ,
		\end{equation}
		because, for all $a \models €p$ and $b\models €q$,
		\begin{align*}
			2\norm{x}^{€p} - 2\norm{x \meet σ^n x}^{€p} &= d(a,σ^n a) \\
			&⩽ d(a,b) + d(b,σ^n b) + d(σ^n b, σ^n a) \\
			&= d(a,b) + 2\norm{x}^{€q} - 2\norm{x \meet σ^n x}^{€q} + d(b, a),
		\end{align*}
		where the equalities follow from the general identity $d(\abs{x}, \abs{y}) = \norm{x} + \norm{y} - 2 \norm{ \abs{x} \meet \abs{y} }$. 
	\end{obs}

	\begin{propos}
		$T_A$ is not $ω$-stable.
	\end{propos}
	\begin{proof}
		We apply the same idea of \cite[Lemma~3.3]{perturbations} to show that there is an uncountable set of $1$-types such that the distance between any two of them is at least $1/2$.
		For every irrational $α>0$, consider the automorphism $σ_{α}$ of $L^p(λ)$ induced by the rotation of $α$ (as defined in \eqref{def rotation}), so that $(L^p(λ),σ_{α})$ is a model of $T_A$, and let $€p_{α}$ be the type of $u = χ\lowersub{\ccint{0,1/2}*}$ therein.
		
		Let $α$ and $β$ be irrational and linearly independent over the rationals. For any $ε>0$, we can find positive integers $n, k, m$ such that $\abs{nα-k} <ε$ and $\abs{n β - m - \frac{1}{2}} <ε$. This means that, up to an error $2ε$, the power $σ_{α}^n$ takes $u$ back to itself, while $σ_{β}^n$ takes $u$ to the other half of $\uint$, that is, $\norm{x \meet σ^n x}^{€p_{α}} > \frac 12 - 2ε$ and $\norm{x \meet σ^n x}^{€p_{β}} < ε$. By \eqref{lower bound on type distance}, we have $d(€p_{α},€p_{β}) > \frac 12 - 3ε$ and thus $d(€p_{α},€p_{β}) ⩾ \frac 12$.
	\end{proof}
	
	We shall now recall the notion of isolation for types of a complete theory $T$ over a countable language.
	
	\begin{defin}
		A type $€p∈S_n(T)$ is \emph{isolated} if for all $r >0$ the ball $B(€p,r)$ contains $€p$ in its topological interior: $€p∈B(p,r)^\circ$ (i.e., the metric and the topology coincide at $€p$).
	\end{defin}
	
	Notice that a type $€p$ is isolated if and only if every net of types topologically converging to $€p$ is also metrically convergent to $€p$. 
	Ryll-Nardzewski Theorem for continuous logic gives a characterisation of separable categoricity in terms of isolated types, namely $T$ is separably categorical if and only if, for every $n$, all $n$-types are isolated. Another interesting fact about isolated types is the following.

	\begin{fact}[{\cite[Corollary~10.10]{claoc}}]\label{omission vs isolation}
		A type $€p∈ S_n(T)$ can be omitted if and only if it is not isolated.
	\end{fact}
	
	\begin{obs}
		In $T_A$, the type of $0$ is isolated. In fact, if $(€p_{α})_{α}$ is a net converging to $\type 0$ in the logic topology, then $\norm{x}^{€p_{α}}$ converges to $\norm{0} = 0$. Now, $d(€p_{α},\type 0) ⩽ \inf§{\norm{a_{α}} : a_{α} \models €p_{α}} = \norm{x}^{€p_{α}}$, hence the convergence in metric.
	\end{obs}
	
	\begin{lemma}\label{no isolated types}
		No non-trivial $1$-type of $T_A$ is isolated.
	\end{lemma}
	\begin{proof}
		Let $€p(x)$ be a non-trivial $1$-type of $T_A$. Then $\norm{x}^{€p}$ is positive, and up to rescaling, we may assume $\norm{x}^{€p} = 1$. We may also suppose that $\norm{x^-}*^{€p} = 0$, for if not we can just replace every occurrence of $x$ in this proof by $\abs{x}$. We shall distinguish two cases:
		\begin{enumerate}
			\item either $\lim_{n} \norm{x \meet σ^n x}^{€p} = 0$
			\item or $\liminf_{n} \norm{x \meet σ^n x}^{€p} > 0$.
		\end{enumerate}
		
		Let $ρ$ be the automorphism of $E_{\uint} = L^p(\uint,£B,λ)$ induced by an irrational rotation on the unit interval $\uint$, as defined in \eqref{def rotation}, and let $τ$ be the automorphism of $E_{\rr} = L^p(\rr,£B,λ)$ induced by a translation of step $-1$ on the real line, as defined in \eqref{def translation}. We have thus two models $(E_{\uint},ρ)$ and $(E_{\rr}, τ)$ of $T_A$.
		
		First notice that $u = χ\lowersub{\ccint{0,1}}$ in $(E_{\rr},τ)$ satisfies $\lim_{n} \norm{u \meet σ^n u} = 0$, while $v = χ\lowersub{\uint}$ in $(E_{\uint},ρ)$ satisfies $\lim_{n} \norm{v \meet σ^n v} = 1$,
		so both cases above may occur and we cannot discard any of them. Now, a type $€p$ in the first case cannot be realised in $(E_{\uint},ρ)$. In fact, if we call $α$ the irrational step of the rotation corresponding to $ρ$, then by Dirichlet's approximation theorem, we can find an increasing sequence $(n_i)_{i<ω}$ such that $n_i α$ is at most $1/n_i$ away from an integer. 
		This means that, by dominated convergence, $\lim_{i} \norm{u \meet ρ^{n_i} u} = \norm{u}$ for any $u ∈ E_{\uint}$, so no element of $E_{\uint}$ can realise $€p$.
		
		Suppose now we are in the second case above and $\liminf_{n} \norm{x \meet σ^n x}^{€p} = α > 0$. Let $u$ be a positive element of $E_{\rr}$ and let $k⩾ 2n+1$. Notice that 
		\begin{equation*}
			\norm{u \meet τ^k u} ⩽ \norm{u \restr {\ooint{-\infty,-n}}} + \norm{u \restr {\ooint{n,+\infty}}} = \norm{f - f \restr \ccint{n,n}} \to 0,
		\end{equation*}
		so $\lim_{n} \norm{u \meet τ^n u} = 0 $, contradicting $α>0$. This shows that $€p$ cannot be realised in $(E_{\rr}, τ)$. 
		
		In both cases, $€p$ can be omitted, so we can conclude using \Cref{omission vs isolation} that it is not isolated.
	\end{proof}
	
	A consequence of this is that $T_A$ does not admit atomic models, i.e., models that only realise isolated types.

	\subsection{Ergodic classification} %
	
	In classical ergodic theory, ergodic non-singular transformations are separated in different \emph{types} based on the existence of a finite or \tsigma-finite equivalent invariant measure.
	We will exploit the correspondence between non-singular transformations and automorphisms of $L^p$ lattices to present a similar classification for the lattices, and we will give a characterisation with a model-theoretic flavour of these types.
	Our main reference here will be \cite{danilenko}.
	
	A \emph{non-singular dynamical system} is an object $(X,£B, μ, τ)$ consisting of a standard Borel space equipped with a \tsigma-finite measure $μ$ and a non-singular transformation $τ$.

	\begin{defin}
		A non-singular system $(X,£B,μ,τ)$ is \emph{ergodic} if every $τ$-invariant $A ∈ £B$ is either negligible or has negligible complement.
	\end{defin}

	$τ$ is said to \emph{preserve a measure} $ν$ on $£B$ if $τ_*ν = ν$, that is, $ν(τ^{-1}A) = ν(A)$ for all $A ∈ £B$. In this case the measure $ν$ is said to be \emph{invariant under $τ$}. 
	
	\begin{defin}
		Suppose that the non-singular system $(X,£B,μ,τ)$ is atomless and ergodic. We say that it is of \emph{\kindii} if there $τ$ preserves a \tsigma-finite measure $ν$ on $£B$ that is equivalent to $μ$, otherwise we say that is of \emph{\kindiii}. When the system is of \kindii\ we make a further distinction: if $τ$ preserves a finite measure equivalent to $μ$ we say that it is of \emph{\kindiif}, otherwise we say that it is of \emph{\kindiin}.
	\end{defin}
	In the literature the term “type” is used instead of “kind” in this classification, but here we prefer to use the latter, so as to avoid confusion with logic types.
	
	\begin{defin}
		A measurable set $W$ is said to be \emph{wandering} if $W ∩ τ^i W = ∅$ for all $i∈\zz$. If we only require the existence of an infinite set $I⊆\nn$ such that $τ^i W ∩ τ^j W = ∅$ for all $i≠j$ in $I$, then $W$ is said to be \emph{weakly wandering}.
	\end{defin}
	
	It turns out that the absence of weakly wandering sets characterises the \kindiif.
	
	\begin{fact}[{\cite[Theorem~1]{hajian}}]\label{hajian-kakutani}
		A non-singular system is of \kindiif\ if and only if it does not admit weakly wandering sets.
	\end{fact}

	An atomless $L^1$ lattice together with an automorphism is called a \emph{lattice system}. We will now define ergodic lattices and their kinds so that in the separable case they correspond to their respective measure-theoretical notions.
	
	\begin{defin}
		We say that a lattice system $(E,σ)$ is \emph{ergodic} if for all positive $u$ and $v$ in $E$, there is $n∈\zz$ such that $u$ and $σ^nv$ are disjoint.
	\end{defin}

	We say that $u$ and $v$ are \emph{compatible} if $u \restr v = v \restr u$. An element $u$ is said to be \emph{autocompatible} if for all $n<ω$, $u$ and $σ^nu$ are compatible.
	
	\begin{defin}
		$(E,σ)$ is of \kindii\ if it admits a positive autocompatible element, otherwise it is of \kindiii. When $(E,σ)$ is of \kindii, if it admits a fixed point, then it is of \kindiif, otherwise of \kindiin.
	\end{defin}

	Suppose now that $E$ is separable, then $(E,σ)$ can be identified with $(L^1(X,£B,μ),\widetilde{τ})$, with $(X,£B,μ)$ a standard atomless probability space, and $τ$ a non-singular transformation of it. Recall that 
	\begin{equation*}
		\widetilde{τ}(u) = \frac{\dd τ_*μ}{\dd μ}  ⋅ §(u \comp τ^{-1})
	\end{equation*}
	for each $u∈ E$.
	
	\begin{propos}
		$(E,σ)$ is ergodic if and only if $(X,£B,μ,τ)$ is ergodic.
	\end{propos}
	\begin{proof}
		Suppose $(E,σ)$ ergodic and let $A∈£B$ be $τ$-invariant. Then, for all $i∈\zz$, $σ^i(χ\lowersub{A})$ is disjoint from $χ\lowersub{X\setdiff A}$, which implies that either $A$ or $X\setdiff A$ is negligible. Conversely, if $(X,τ)$ is ergodic and $f,g∈ E$ are positive, then the support of $f$ is included modulo $μ$ in $\bigcup_{i∈\zz} τ^i\supp(g)$, as the the latter is $τ$-invariant. This means that there is $i∈\zz$ such that $μ§(\supp(f) ∩ \supp(σ^i g)) > 0$, showing that $(E,σ)$ is ergodic.
	\end{proof}
	
	\begin{propos}
		Let \symx\ be either \symiif, \symiin, or \symiii. Then, $(E,σ)$ is of kind \symx\ if and only if $(X,£B,μ,τ)$ is of kind \symx.
	\end{propos}
	\begin{proof}
		Suppose $(E,σ)$ is of \kindii.
		Then there exists a positive autocompatible element $f∈E$. Define $A_i = \supp(σ^i f)$ and $g = \sup_{i∈\zz} σ^i f$. Then $g \restr A_i = σ^i f$, and by ergodicity, $\bigcup_i A_i = X$ modulo $μ$, so $g$ is positive almost everywhere and thus $ν(A) = \int_A g \dd μ$ is a measure on $X$ equivalent to $μ$. Clearly, $ν$ is \tsigma-finite because $ν(A_i) = \norm{σ^if}* = \norm{f}$, which is finite.
		
		For the converse, suppose $τ$ preserves a \tsigma-finite measure $ν$ equivalent to $μ$. Then the Radon--Nikodym derivative $g= \frac{\dd μ}{\dd ν}$ is a measurable function positive almost everywhere, and there is some $A∈£B$ such that $f = g\restr A$ has finite integral, so that $f∈ E$. 
		
		It is easy to check that if $μ$ and $ν$ are equivalent \tsigma-finite measures and $τ$ is non-singular with respect to $μ$, or equivalently to $ν$, then 
		\begin{equation}\label{identity rn pf}
			\frac{\dd τ_*μ}{\dd τ_*ν} \comp τ = \frac{\dd μ}{\dd ν},
		\end{equation}
		from which it follows
		\begin{equation*}
			\widetilde{τ}^n f = \frac{\dd τ^n_*μ}{\dd μ} ⋅ (f \comp τ^{-n}) = \widetilde{τ^n} f,
		\end{equation*}
		for all integer $n$.
		Let $x ∈ A ∩ τ^{i} A$, then by \eqref{identity rn pf},
		\begin{equation*}
			f(τ^{-i} x) = \frac{\dd ν}{\dd μ} (τ^{-i} x) 
			=  \frac{\dd ν}{\dd τ^i_*μ} ( x)
		\end{equation*}
		so that 
		\begin{math}
			σ^i f(x) =  \frac{\dd τ^i_* μ}{\dd μ}  (x) f(τ^{-i} x) = f(x),
		\end{math}
		showing that $f$ is autocompatible.
		
		The same argument, with $A_0 = A = X$, shows that 
		$(E,σ)$ is of \kindiif\ if and only if $(X,£B,μ,τ)$ is of \kindiif, which concludes the proof.
	\end{proof}

	Let $€p$ be a $1$-type in $S_{x}(T_{σ})$ satisfying $\norm{x}^{€p} = \norm{x^+}*{}^{€p} = 1$, that is, the type of a positive element with norm $1$. We say that $€p$ is a \emph{type of fixed point} if $\norm{σx - x}^{€p} = 0$, a \emph{type of compatibility} if $\norm{\frac{σ^nx}{2} \meet x - \frac x2 \meet σ^n x}^{€p} = 0$ for all $n∈\nn$, a \emph{type of weak wandering} if there is an infinite set $I⊆\nn$ such that $\norm{σ^i x \meet σ^j x}*^{€p} = 0$ for all $i≠j$ in $I$.
	
	As a corollary to \Cref{hajian-kakutani}, we have that $(E,σ)$ is of \kindiif\ if and only if it does not realise a type of weak wandering.
	This lets us distinguish separable lattice systems of different kinds based on the $1$-types they realise.
	\begin{propos}
		A separable lattice system $(E,σ)$ is
		\begin{itemize}[nosep, label={---}]
			\item of \kindiif\ when it realises a type of fixed point, but no type of weak wandering,
			\item of \kindiin\ when it realises a type of compatibility, a type of weak wandering, but no type of fixed point,
			\item of \kindiii\ when it realises a type of weak wandering, but no type of compatibility.
		\end{itemize}
	\end{propos}
	
	\begin{question}
		Is there a non trivial $1$-type which is realised by two lattice systems of different kinds?
	\end{question}

	\section{Conjugacy classes of $\aut^*(μ)$} \label{section: conjugacy}
	
	Let $(X, £B, μ)$ be a standard atomless probability space and let $\aut^*(μ)$ denote the group of all non-singular transformations of $X$. Then $E = L^1(X,£B, μ)$ is a separable model of \alol, and the group $\aut (E)$ of automorphisms of Banach lattices is isomorphic to $\aut^*(μ)$ via \eqref{action def}. As we have seen in \Cref{rokhlin section}, this isomorphism identifies the subset $S⊆\aut^*(μ)$ of aperiodic transformations with the the set of those automorphisms $σ ∈ \aut( E)$ that make $(E,σ)$ a model of $T_A$. For the sake of simplicity, in the following we will not distinguish between an element of $\aut^*(μ)$ and the induced automorphism of $E$.
	
	In \cite{tulcea}, Ionescu Tulcea introduced the \emph{weak topology} $\wtopol$ of $\aut^*(μ)$, which can be described as the topology of pointwise convergence (or the strong operator topology) in $\aut(E)$, transferred to $\aut^*(μ)$ via the isomorphism above. This means that if $(σ_n)_n$ is a sequence in $\aut^*(μ)$, then it converges to some $σ_{*}$ in $\wtopol$ if and only if $\norm{\widetilde{σ}_nu - \widetilde{σ}_{*}u} \to 0$ for all $u∈E$.
	This topology makes $\aut^*(μ)$ a Polish group and it does not change if we replace $μ$ with another equivalent measure.

	The topological properties of $\aut^*(μ)$ have been extensively studied. For a thorough treatment, we refer the reader to \cite{danilenko, halmos, friedman}. The following fact is a fundamental tool in this context.
	
	\begin{fact}[{\cite[p.\ 77]{halmos}}]\label{conjugacy lemma}
		The conjugacy class of each aperiodic transformation is dense in $\aut^*(μ)$.
	\end{fact}

	For any formula $φ$ we denote its interpretation in the structure $(E,σ)$ by $φ^{σ}$.
	We denote the type of $\seq{u}$ in $(E,σ)$ by $\type^{σ}(\seq u)$
	
	\begin{lemma}\label{char topology in aut}
		The weak topology on $\aut^{*}(μ)$ is precisely the initial topology with respect to the family of functions $σ \mapsto φ^{σ}(u)$, where $φ(x)$ is a quantifier-free formula in $£L_{σ}$ and $u ∈ E^{\abs{x}}$.
	\end{lemma}
	\begin{proof}
		Let $(σ_{α})_{α}$ a net in $\aut^{*}(μ)$ converging to some $σ_{*}$ in the initial topology described above, and consider $φ(x,y) = \norm{σx-y}$. Then, for any $u ∈ E$,  $φ^{σ_{α}}(u,σ_{*}u) = \norm{σ_{α} u - σ_* u}$ converges to $φ^{σ_*}(u,σ_{*}u) = 0$, showing that $σ_{α} \smash{\xrightarrow{\wtopol}} σ_*$.
		
		Conversely, suppose $(σ_{n})_{n}$ is a sequence converging to $σ_*$ in the weak topology, and let $φ(x)$ be a quantifier-free formula in $£L_{σ}$ and $\seq u ∈ E^{\card{x}}$. 
		
		We can rewrite $φ(x)$ as $ψ(x,σx,…,σ^k)$, with no instance of $σ$ occurring in $ψ(x,y_1,…,y_k)$. For each $i<k$,
		\begin{equation*}
			\norm{σ_n^{i+1} u - σ_*^{i+1} u} ⩽ \norm{σ_n^i u - σ_*^i u} + \norm{σ_n (σ_*^i u) - σ_* (σ_*^i u)},
		\end{equation*}
		which goes to zero by induction and $\wtopol$-convergence. 
		As $ψ^E$ is uniformly continuous, we conclude that $φ^{σ_n}(\seq{u})$ converges to $φ^{σ_{*}}(\seq{u})$.
	\end{proof}
	
	\begin{obs}\label{topology of S}
		As $T_A$ eliminates quantifiers, the restriction of $\wtopol$ to $S$ can be described as in the previous lemma, but with $φ$ varying among all formulas.
	\end{obs}
	
	It is easy to see that the sets $\doublesqbr{φ^{\cdot}u >0} = §{σ∈\aut^*(μ) : φ^{σ}u >0}$, where $φ(x)$ is a quantifier-free formula in $£L_{σ}$ and $u ∈ E^{\abs{x}}$, form a basis for the weak topology.

	\begin{lemma}\label{s is polish}
		$S$ is a dense \gdelta\ subset of $\aut^*(μ)$. In particular, $S$ is a Polish space.
	\end{lemma}
	\begin{proof}
		As the conjugate of an aperiodic transformation is still aperiodic, $S$ is invariant under conjugation. This means that if $σ∈S$ and we denote its conjugacy class by $\class{σ}$, then $\class{σ}⊆S$, but by \Cref{conjugacy lemma} $\class{σ}$ is dense in $\aut^*(μ)$, so $S$ is too.
		
		We will now show that $S$ is \gdelta. Since the elements of $S$ are precisely those that satisfy Rokhlin axioms $\sup_x \inf_y \rokh{n}(x,y)$ (as defined in \eqref{rokhlin axiom}), we can rewrite $S$ as 
		\begin{equation*}
			S = \bigcap_{n<ω} \bigcap_{u ∈ E_0} \bigcap_{m>0} \bigcup_{v∈E} §{σ ∈ \aut^*(μ) : \rokh{n}^{σ}(u,v) < 1/m},
		\end{equation*}
		where $E_0$ is a countable dense subset of $E$. By the previous lemma, $§{σ : \rokh{n}^{σ}(u,v) < 1/m}$ is open, and thus $S$ is countable intersection of open sets.$  $
	\end{proof}

	Let $φ(x)$ be a formula in $£L_{σ}$. The family of its interpretations $φ^{σ}\colon E^{\card{x}} \to \rr$ with $σ$ varying in $S$ is equicontinuous, so the map
	$π_{φ} \colon S × E^{\card{x}} \to \rr$ defined by $π_{φ}(σ,u) = φ^{σ}(u)$,
	is continuous, using the characterisation in \Cref{topology of S}. This means that the maps
	\begin{align*}
		θ_{n} \colon S × E^n &\to S_n(T_A)\\
		(σ,u) &\mapsto \type^{σ}(u)
	\end{align*}
	are also continuous.
	We will also make use of the following fact, which is an easy consequence of quantifier elimination for $T_A$.
	\begin{fact}\label{type of conj}
		$φ^{fσf^{-1}}(u) = φ^{σ}(f^{-1}u)$ and $\type^{fσf^{-1}}(u) = \type^{σ}(f<^{-1}u)$ for any $f∈\aut(E)$ and $σ∈S$, by quantifier elimination.
	\end{fact}

	Recall that the \emph{thickening} of a set $A$ by $r$ in $S_{n}(T_A)$ is the set $B(A,r) = \bigcup_{€p ∈ A} B(€p, r)$. In $S_{n}(T_A)$, topological openness is preserved by thickening \cite[Lemma~10.2]{claoc}. 
	The following lemma, whose proof is largely due to Todor Tsankov, guarantees that for any $v∈E^n$, the image under $θ_n(\,⋅\, , u)$ of an open set in $S$ has open thickenings in $S_n(T_A)$.

	\begin{lemma}\label{open thickening}
		If $U⊆S × E^n$ is open, then for all $r>0$, there thickening of $θ_{n}(U)$ by $r$ is open in the logic topology.
	\end{lemma}
	\begin{proof}
		Let $(σ,u) ∈ U$. As $U$ is open, we can find $φ(y) ∈ £L_{σ}$, $w∈E^{\card{y}}$ and $ρ>0$ such that
		\begin{equation*}
			(σ,u) ∈ \doublesqbr{φ^{\cdot}w >0} × B(u,ρ) ⊆ U.
		\end{equation*}
		Without loss of generality, we may assume that $φ^{σ}(w)=2$. To show that $B(θ_n(U),r)$ is open, it suffices to find some open neighbourhood $V$ of $\type^{σ}(u)$ such that for all $€q∈V$ there are $τ ∈ \doublesqbr{φ^{\cdot}w >0}$ and $v ∈ B(u,ρ)$ such that $d(€q,\type^{τ}v) < r$.
		Since the family $§{φ^{σ} : σ ∈ S}$ is equicontinuous, there is a positive $δ$ such that
		\begin{equation*}
			\abs{φ^{τ}(w') - φ^{τ}(w)} < 1
		\end{equation*}
		for all $τ∈S$ and all $w' ∈ B(w,2δ)$. We may also require $2δ ⩽ r$. 
		
		As \alpl\ is separably categorical, the logic and metric topologies coincide \cite[\S~4.3]{cont-logic}. This means that the ball $B(\type^E(uw),δ)$ contains a topological neighbourhood of $\type^E(uw)$, so there is a formula $ψ(x,y)$ in $£L_{\mathrm{BL}}$ such that $ψ^E(u,w) = 2$, and for all $u'w' ∈ E^{\card{xy}}$ satisfying $ψ^E(u',w') >0$, we have $d(\type^E(u'w'), \type^E(uw)) < δ$. This in turn implies that there is some $u''w'' \sametype^E u'w'$ such that $d(u''w'', u w) < δ$, by \cite[Proposition~10.4]{claoc}. As $E$ is a separable model of \alpl, by \cite[Prop.~9.8 and Corollary 10.12]{claoc} it is approximately homogeneous, so there exists an automorphism $f$ of $E$ such that $d(u''w'', f(u'w')) < δ$, and thus $d(f(u'w'),uw) < 2δ$ by the triangle inequality.
		
		Define $V = \doublesqbr{ \sup_{y}(φ(y) \meet ψ(x,y)) > 1 }$. This is an open neighbourhood of $\type^{σ}(u)$, because $φ^{σ}(w) = ψ^E(u,w) = 2$. Now let $€q ∈ V$. There are $σ_0∈ S$ and $u' ∈ E^n$ such that $€q = \type^{σ_0}(u')$, and $w' ∈ E^{\card{y}}$ such that $φ^{σ_0} (w') > 1$ and $ψ^E(u',w') >1$. By the previous paragraph, we can find $f∈ \aut(E)$ such that
		\begin{equation}\label{proof: distance}
			d(f(u'w'),uw) < 2δ. 
		\end{equation}
		In particular, $d(fw',w) < 2δ$, so, by the choice of $δ$, we have $\abs{φ^{τ}(fw') - φ^{τ}(w)} < 1$ for all $τ∈S$. 
		
		Now let $τ = fσ_0f^{-1}$. By \Cref{type of conj}, $φ^{τ}(fw') = φ^{σ_0} (w') >1$, so the last inequality in the previous paragraph yields $φ^{τ}(w) > 0$, implying that $τ ∈  \doublesqbr{φ^{\cdot}w >0}$. If we choose $v= u$, then the condition $v∈ B(u,ρ)$ is trivially satisfied, and we just need to check that $d(€q, \type^{τ}u) < r$. Again by \Cref{type of conj}, we have $€q= \type^{τ}(fu')$. Therefore,
		\begin{equation*}
			d(€q, \type^{τ}u) = d(\type^{τ}(fu'), \type^{τ}u) ⩽ d(fu',u) < 2δ ⩽ r,
		\end{equation*}
		by \eqref{proof: distance}, which concludes the proof.
	\end{proof}

	As the distance in the topometric type space $(S_n(T),\typetop,\partial)$ is lower semicontinuous, the set $§{(€p,€q) ∈ S_n(T)^2 : d(€p,€q) ⩽ r}$ is closed, so its sections are too, but these are precisely the closed balls of radius $r$ in $S_n(T)$. We have thus the following fact.
	\begin{fact}\label{closed balls}
		Metrically closed balls in a type space are also topologically closed.
	\end{fact}

	\begin{lemma}\label{comeagre omission}
		Let $€p$ be a non-isolated $n$-type of $T_A$. Then there is a comeagre subset $S_0$ of $S$ such that $(E,σ)$ omits $€p$ for all $σ ∈ S_0$. 
	\end{lemma}
	\begin{proof}
		By \cite[Lemma~10.3]{claoc} there is a ball $B(€p,2r)$ around $p$ with empty interior. By \Cref{closed balls}, the closed ball $B[€p,r]$ is closed in the logic topology. This means that the preimage $C = θ_n^{-1} B[€p,r]$ is closed in $S×E^n$, by continuity of $θ_n$. We will show that $C$ has empty interior.
		Suppose we have an open subset $U$ of $C$. By \Cref{open thickening}, the thickening of $θ_n(U)$ by $r$ is open in the logic topology, but it is also included in the ball $B(€p,2r)$, which has empty interior. This means that $U=∅$ and thus $C$ has empty interior.
		
		The previous paragraph also shows that the set $A = θ_{n}^{-1}(B(€p,r)$ is nowhere dense in $S×E^n$. By Kuratowski--Ulam theorem, there exists a comeagre subset $S_0$ of $S$ such that, for all $σ∈S_0$ the set $A_{σ} = §{u∈E^n: (σ,u) ∈ A} = §{u∈E^n: \type^{σ}(u) ∈ B(€p,r)}$ is nowhere dense. 
		
		Consider now the map $t_{σ}\colon E^n \to S_n(T_A)$ defined by $t_{σ}(u) = θ_n(σ,u) = \type^{σ}(u)$, and notice that this is continuous with respect to the type metric, since $d(\type^{σ}u, \type^{σ}v) ⩽ d(u,v)$. We can now rewrite $A_{σ}$ as $t_{σ}^{-1}B(€p,r)$, showing that $A_{σ}$ is open in $E^n$, which implies that it is nowhere dense precisely when it is empty.	This means that for every $σ∈S_0$, the model $(E,σ)$ omits $€p$.	
	\end{proof}
	
	\begin{theo}
		No conjugacy class in $\aut^*(μ)$ is comeagre.
	\end{theo}
	\begin{proof}
		Suppose on the contrary that there is some $σ∈G = \aut^*(μ)$ such that $\class{σ}$ is comeagre in $G$. 
		As $S$ is comeagre in $G$, the intersection $\class{σ} ∩ S$ is also comeagre in $G$, and thus non-empty. Since $S$ is invariant under conjugation, we deduce that $\class{σ} ⊆ S$. By \Cref{s is polish}, $S$ is Polish, so $\class{σ}$ is in fact comeagre in $S$.
		
		Let $u$ be a non-zero element of $E$ and consider its type $€p(x) = \type^{σ}(u)$ in $(E,σ)$. As $€p$ is a non-trivial $1$-type, it is not isolated, by \Cref{no isolated types}.
		It follows from \Cref{comeagre omission} that there are comeagrely many $τ∈S$ such that the model $(E,τ)$ omits $€p$. Given that the intersection of comeagre sets is non-empty, there exists a conjugate $fσf^{-1}$ of $σ$ such that no $u∈E^{\card{x}}$ satisfying $\type^{fσf^{-1}}(u) = €p$. But $\type^{fσf^{-1}}(u) = \type ^{σ}(f^{-1}u)$ by \Cref{type of conj}, hence $(E,σ)$ omits $€p$, a contradiction.
	\end{proof}
	
	\begin{corol}
		Every conjugacy class in $\aut^*(μ)$ is meagre.
	\end{corol}
	\begin{proof}
		By invariance of $S$ under conjugation, a class $[σ]$ is either included in $S$ or in its complement. In the first case, it is dense, so the Effros theorem and the previous result imply that it is meagre. In the second case, it is meagre because $S$ is comeagre.
	\end{proof}

	\medskip
	\nocite{hodges}
	\printbibliography
\end{document}